\documentclass[11pt]{amsart}
\usepackage{amscd,amssymb,amsmath,latexsym,enumerate}
\usepackage[mathscr]{euscript}
\usepackage{epsfig}

\usepackage{color}

\newtheorem{theo}{Theorem}[section]
\newtheorem{proposi}[theo]{Proposition}
\newtheorem{lemma}[theo]{Lemma}
\newtheorem{coro}[theo]{Corollary}
\theoremstyle{definition}
\newtheorem{exam}[theo]{Example}
\newtheorem{rem}[theo]{Remark}
\newtheorem{defini}[theo]{Definition}

\newcommand{\Pp}{{\mathcal P}}

\newcommand{\Ww}{{\mathcal W}}

\newcommand{\CM}{{\mathbb C}}

\newcommand{\NM}{{\mathbb N}}

\newcommand{\RM}{{\mathbb R}}

\newcommand{\TM}{{\mathbb T}}

\newcommand{\ZM}{{\mathbb Z}}

\newcommand{\as}{{\mathscr A}}

\newcommand{\sop}{\sigma^{\mathrm{op}}}               
\newcommand{\spec}{\mathrm{spec}}                     
\newcommand{\specess}{\mathrm{spec}_{\mathrm{ess}}}   
\newcommand{\specpoi}{\mathrm{spec}_{\mathrm{point}}} 
\newcommand{\Orb}{\textit{Orb}}                       
\newcommand{\Band}{\mathcal{BO}}                      
\newcommand{\BDO}{\mathcal{BDO}}                      
\newcommand{\PE}{{\rm \Psi E}}                        

\begin{document}

\title[]{Note on  spectra of non-selfadjoint operators over dynamical systems}
\author{Siegfried Beckus, Daniel Lenz, Marko Lindner, Christian Seifert}

\address{Mathematisches Institut\\
Friedrich-Schiller-Universit\"at, Jena\\
07743, Jena, Germany}
\email{siegfried.beckus@uni-jena.de}
\email{daniel.lenz@uni-jena.de}


\address{Technische Universit\"at Hamburg-Harburg\\
Institut f\"ur Mathematik\\
21073 Hamburg, Germany}
\email{marko.lindner@tuhh.de}
\email{christian.seifert@tuhh.de}


\begin{abstract} We consider   equivariant continuous  families  of
discrete one-dimensional operators over arbitrary  dynamical
systems. We introduce the concept of a pseudo-ergodic element of a
dynamical system. We then  show that all  operators associated to
pseudo-ergodic elements have the same spectrum and that this spectrum
agrees with their essential spectrum. As a consequence we obtain
that the spectrum is constant and agrees with the essential spectrum
for all elements in the dynamical system if minimality holds.
\end{abstract}


\maketitle

\section*{Introduction}
Selfadjoint random operators arise in the quantum mechanical
treatment of disordered solids. Their study has been a key focus of
mathematical physics in the last four decades. Indeed, in an
impressive number of (classes of) specific examples explicit
spectral features (such as pure point spectrum  or purely singular
continuous spectrum or Cantor spectra) could be proven, see e.g. the
surveys and monographs \cite{Dam,DEG,CarLac,Lenz1,PF,Sto}.

A particularly rich class of examples has been treated in one
dimension. Corresponding models arise mostly  by codings of
topological dynamical systems via sample functions.

A very basic result in this context is constancy of the spectrum
provided the underlying dynamical system is minimal and the
selfadjoint operators satisfy a weak continuity condition. In fact,
this constancy of the spectrum has been (re)proven in various works.
For almost periodic operators it can be inferred from \cite{Joh},
see Chapter 10 of \cite{CFKS}  as well. For special quasicrystal
operators a statement is contained in \cite{BIST}. A rather general
result for minimal systems is then discussed in \cite{Lenz2}. In any
case, the constancy  is a rather direct consequence of a
semicontinuity property of the spectrum of selfadjoint operators.

Now, recent years have seen quite some interest in non-selfadjoint
random type  operators, see e.g.
\cite{BGS,BES,DNS,Davies2001b,GK,Mar,NS,FeinZee99b,HatanoNelson96} and references therein. In
this context many spectral questions are wide open. In fact, even
the most basic issue of constancy of the spectrum of operators
associated to minimal dynamical systems can not be inferred
immediately as the basic argument from the selfadjoint case
completely breaks down. The reason for this break down  is that the
spectra of non-selfadjoint operators do not have a semicontinuity
property (as is well-known, see  e.g. \cite[Example
IV.3.8]{Kato1980}, compare Section \ref{ConstSpectr.sect-OpDynSys}
below as well).

At the same time the concept of pseudo-ergodicity has been brought
forward in \cite{Davies2001b} (and has been successfully employed
since, see e.g. \cite{Lindner2006,CL10,CCL13,LindnerRoch12}) in
the context of non-selfadjoint operators in order to deal with random
examples without having to worry about a stochastic component. In this
context, some version of constancy of the spectrum  could be shown.
However, this does not give constancy of the spectrum for all involved
operators but only among those satisfying  the pseudo-ergodicity condition.

The aim of the present  note is to reconcile these different points
of view. Specifically, we introduce the concept of a pseudo-ergodic
element of an arbitrary  dynamical system in Section
\ref{ConstSpectr.sect-BackDynSys} as well as  the setting of
equivariant operator families over a dynamical system  in Section
\ref{ConstSpectr.sect-OpDynSys}. We then combine these considerations
to obtain our main abstract result in Section \ref{section-main}.
This result, Theorem \ref{ConstSpectr.theo-MinConstSpectr}, gives
constancy of the spectrum among the pseudo-ergodic elements of the
dynamical system. As discussed in Section
\ref{section-example-Davies}, this generalizes the result of
\cite{Davies2001b} (in the case that the underlying group is $\ZM$).
If, on the other hand, the dynamical system is minimal then all
elements turn out to be pseudo-ergodic and constancy of the spectrum
for all involved operators follows, Corollary \ref{main-minimal}.
This corollary extends to the non-selfadjoint case the results
mentioned above. In Section \ref{section-examples} we present some
examples of minimal systems which are heavily studied in the
selfadjoint case. We also indicate there some non-selfadjoint
operators of interest to which the corollary can be applied.

As this discussion shows,  Theorem
\ref{ConstSpectr.theo-MinConstSpectr} can be seen as a
generalization of both the result mentioned above  for selfadjoint
operators in the minimal case and the result mentioned above for
non-selfadjoint operators in the pseudo-ergodic case.    Along the
way we will also show that the spectrum agrees with the essential
spectrum (which is also known in the selfadjoint case).

\smallskip

The considerations below are phrased in the setting of dynamical
systems over $\ZM$. This is for convenience mostly. Indeed, the
underlying theory of dynamical systems is valid for substantially
more general systems over a discrete countable group $\varGamma$.
Thus, our main result can be carried over to such systems whenever a
suitable version of Theorem \ref{ConstSpectr.theo-EssSpectLocInfin}
is at hand.

\medskip

\textbf{Acknowledgments.} D.L. gratefully acknowledges inspiring
discussions with Lyonell Boulton and the stimulating atmosphere of
the 'Workshop on Advances and Trends in Integral Equations' in 2009,
where part of this work was conceived. M.L. and C.S. wish to thank
Raffael Hagger for very helpful comments.

\section{Background: Classes of operators}
\label{ConstSpectr.sect-BackOper}

Let $\NM$ be the set of all positive integers and $\ZM$
the set of all integers. Then, for $p\in[1,\infty)$, $\ell^p:=\ell^p(\ZM)$ denotes the
space of all two-sided infinite sequences $f:\ZM\to\CM$ such that
$\sum_{j\in\ZM} |f(j)|^p$ is finite. Moreover, $\ell^\infty:=\ell^\infty(\ZM)$ is
the set of all two-sided infinite, bounded sequences.

\medskip

A matrix $A:\ZM\times\ZM\to\CM$ is called a {\em band
matrix} if the following two conditions hold.
\begin{itemize}
\item[$(i)$] The map $A$ is bounded, i.e. $\sup_{i,j\in\ZM}|A_{i,j}|<\infty$.
\item[$(ii)$] There exists a {\em band-width} $w\in\NM$ such that $A_{i,j}=0$ for all $i,j\in\ZM$ satisfying $|i-j|>w$, i.e. $A$ has finitely many non-zero diagonals only.
\end{itemize}

Any band matrix $A$ generates a linear operator $A$ on
each $\ell^p$, $p\in[1,\infty]$, by
$$
\left(Af\right)(i)=\sum\limits_{j\in\ZM} A_{i,j} f(j),\qquad i\in\ZM,\; f\in \ell^p.
$$
Since a band matrix has a finite band-width, this sum is
always finite. Thus, the operator is well-defined. Moreover, we
immediately deduce that $A$ is a bounded operator on the space
$\ell^p$, $p\in[1,\infty]$. Such an operator is called a
{\em band operator}. In the following we will not distinguish
between the matrix and the operator. Furthermore, the set of all
band operators is denoted by $\Band$.

In the literature the matrix of the operator $A$ is often called the
{\em matrix representation}. One could define a norm by
$$
\|A\|_\Ww:=\sum\limits_{k\in\ZM}\sup\limits_{j\in\ZM} |A_{j+k,j}|
$$
on the set of band operators. The closure $\Ww$ of the band
operators $\Band$ with respect to this norm $\|\cdot\|_\Ww$ is
called the {\em Wiener algebra}.

Let $p\in[1,\infty]$ be given. Then $\Band$ is a subset of
$L(\ell^p)$, the bounded linear operators on $\ell^p$. Note that
$\Band$ is not closed in $L(\ell^p)$. Let $\BDO(\ell^p)$ be the
closure of $\Band \subseteq L(\ell^p)$. These operators are called
\emph{band-dominated operators}.

For $m\in\NM_0$ define $P_m$ to be the operator of multiplication by
the characteristic function of ${\{-m,\ldots,m\}} $ and
$\Pp:=\{P_m:\; m\in\NM_0\}$.  We then  set (see e.g.~\cite[Section
1.1]{RabinovichRochSilbermann2004})
\[L(\ell^p,\Pp):=\{A\in L(\ell^p):\; \forall\,m\in\NM_0: \lim_{n\to\infty} \|P_mA(I-P_n)\|+\|(I-P_n)AP_m\| = 0\}.\]

Note that (see e.g. Section 1.3.7 in \cite{Lindner2006})
\[\Band \subseteq \Ww \subseteq \BDO(\ell^p)\subseteq L(\ell^p,\Pp)\subseteq L(\ell^p)\]
for all $p\in[1,\infty]$. Moreover, $L(\ell^p,\Pp) = L(\ell^p)$ for $1<p<\infty$
(but this equality fails for $p=1$ or $p=\infty$). All three,
$(\Ww,\|.\|_\Ww)$, $(\BDO(\ell^p),\|.\|)$ and $(L(\ell^p,\Pp),\|.\|)$ are
Banach algebras that are closed under passing to the inverse operator
(see e.g. Theorems 1.1.9, 2.1.8 and 2.5.3 in \cite{RabinovichRochSilbermann2004}).

Denote by $U$ the shift operator on the set of two-sided infinite
sequences with values in $\CM$, i.e.\ $(Uf)(k):=f(k-1)$ for all
$k\in\ZM$ and $f:\ZM\to\CM$. Its inverse is given by
$(U^{-1}f)(k):=f(k+1)$ for all $k\in\ZM$. Clearly, $U$ descends to
an isometric bijective operator on any $\ell^p$.

For $A = (A_{i,j})_{i,j\in \ZM}\in L(\ell^p,\Pp)$ with $p\in[1,\infty]$,
we will now look at partial limits (with respect to matrix-entrywise
convergence) of the operator sequence $(U^{-n}AU^n)_{n\in\ZM}$:
A matrix $B:\ZM\times\ZM\to\CM$ is called a {\em limit operator induced
by the operator $A$} whenever there exists a sequence $(h_k)_{k\in\NM}$
of integers such that $\lim_{k\to\infty}|h_k|=\infty$ and
$$
B_{i,j}=\lim\limits_{k\to\infty} A_{i+h_k,j+h_k},\qquad i,j\in\ZM.
$$
Clearly, if $A$ is a band operator with band width $w\in \NM$ then a
limit operator $B$ induced by $A$ is a band operator as well.
Furthermore, its band width is smaller or equal to $w$.
Similarly, $B$ belongs, respectively, to $\Ww$, $\BDO(\ell^p)$
or $L(\ell^p,\Pp)$ if $A$ does.

We define by $\sop(A)$ the set of all limit operators induced by
$A$, which is sometimes called the \textit{operator spectrum of
$A$}. By \cite[Corollary 3.24]{Lindner2006}, we have $\sop(A)\neq
\varnothing$ for $A\in \BDO(\ell^p)$. An operator $A\in
L(\ell^p,\Pp)$ is called {\em self-similar} if $A\in\sop(A)$ holds.

\medskip

For $p\in[1,\infty]$ and $A \in L(\ell^p,\Pp)$ we write
$\spec(A)$ for the spectrum of $A$ and $\specess(A)$ for the
essential spectrum of $A$.

\medskip

With this notation at hand the following theorem holds, as shown in
\cite{Seidel2014}.

\begin{theo}[{\cite[Corollary 12 and Theorem 16]{Seidel2014}}]
\label{ConstSpectr.theo-EssSpectLocInfin} Let $p\in[1,\infty]$ and
$A \in L(\ell^p,\Pp)$. Then
$$
\specess(A)\supseteq\underset{B\in\sop(A)}{\bigcup}\spec(B)
$$
holds. In particular,
$$
\spec(A)=\specess(A)
$$
holds if $A$ is self-similar.
\end{theo}

For band-dominated operators, one can even obtain an equality, see \cite{LindnerSeidel2014}, and also \cite{LangeRabinovich1985,RabinovichRochSilbermann1998,Lindner2003,RabinovichRochSilbermann2004,CL10}
for earlier versions.

\begin{proposi}[{\cite[Corollary 12]{LindnerSeidel2014}}]
\label{ConstSpectr.proposi-EssSpectLocInfin} Let $p\in[1,\infty]$ and
$A \in \BDO(\ell^p)$. Then
$$
\specess(A)=\underset{B\in\sop(A)}{\bigcup}\spec(B).
$$
\end{proposi}

As it turns out the spectrum and the essential spectrum of
$A\in \Ww$ are independent of $p\in[1,\infty]$, see
\cite{Kurbatov,Li:Wiener}.
We will use the notation $\specpoi^\infty(A)$ for the set of
eigenvalues of the operator $A$ on the space $\ell^\infty$.

\begin{proposi}[{\cite[Theorem 3.1]{CL08}}]
\label{prop:Spectrum_Wiener} Let $A\in \Ww$. Then
\[\specess(A)=\underset{B\in\sop(A)}{\bigcup}\spec(B) =
\underset{B\in\sop(A)}{\bigcup}\specpoi^\infty(B).\]
\end{proposi}

\begin{rem}  For special selfadjoint operators on $\ell^2$
related results are contained in \cite{LaSi}, see \cite{CarLac,PF}
as well for related results in the context of random selfadjoint
operators.
\end{rem}

For the remaining part of the paper, we fix $p\in[1,\infty]$.

\section{Background: Dynamical Systems}
\label{ConstSpectr.sect-BackDynSys}

Let $(X,T)$ be a dynamical system, i.e.\ $X$ is a compact metric
space and $T:X\to X$ is a homeomorphism. Then, for $x\in X$ the
limit sets $L^+(x)$ and $L^-(x)$ are defined by
$$
L^\pm(x):=\left\{y\in X\;:\; \text{there exists } h_k\to\infty
\text{ such that } \lim\limits_{k\to\infty} T^{\pm h_k} x=y\right\}.
$$
Moreover, the $\textit{orbit}$ of $x\in X$ is given as
$$\Orb(x):=\{ T^n x: n\in \ZM\}.$$

\begin{rem}
In the literature often  the limit set $L^+(x)$ is called the
$\omega$-limit set and the limit set $L^-(x)$ is called the
$\alpha$-limit set of $x$.
\end{rem}

The following proposition is well-known. We include a proof for the
convenience of the reader.

\begin{proposi}
\label{ConstSpectr.prop-PropLimSet} Let $(X,T)$ be a
dynamical system. Then, for any $x\in X$ the sets $L^\pm(x)$ are
non-empty, compact and invariant under $T$ and $T^{-1}$.
\end{proposi}

\begin{proof}
Let $x\in X$. We will give a proof for the set $L^+(x)$ only.
Analogously, the statement for $L^-(x)$ can be proven.

\medskip

By compactness of $X$ the sequence $(T^n x)_{n\in\NM}$ has a convergent subsequence.
Thus, the set $L^+(x)$ is non-empty. Let
$y=\lim_{k\to\infty}T^{h_k}x\in L^+(x)$ be arbitrary. Since $T$ is a
homeomorphism the limits $\lim_{k\to\infty}T^{h_k\pm 1}x$ exist and
they are equal to $T^{\pm 1}y$. Hence, $Ty$ and $T^{- 1}y$ are
elements of $L^+(x)$ implying that $L^+(x)$ is $T$ and
$T^{-1}$-invariant.

\medskip

We now turn to proving the compactness of $L^+(x)$. As $X$ is a
compact metric space it suffices to show closedness. Thus, let
$(y_n)_{n\in\NM}$ in $L^+(x)$ be convergent to $y$ in $X$. For each
$n\in\NM$ there exists $(h^n_k)_{k\in\NM}$ such that
$y_n=\lim_{k\to\infty} T^{h^n_k}x$.
%
%
Thus, there exists a subsequence $(k_n)_{n\in\NM}$ such that for
$(\tilde{h}_n)_{n\in\NM}:= (h^n_{k_n})_{n\in\NM}$ we have
$\tilde{h}_n\to \infty$ and $T^{\tilde{h}_n}x\to y$. Consequently,
$L^+(x)\subseteq X$ is closed.
\end{proof}

We will be interested in those elements of $X$ for which the union of
the limit sets agrees with $X$.

\begin{defini}[Pseudo-ergodic elements of $(X,T)$] Let $(X,T)$ be
a dynamical system. An $x\in X$ is called \emph{pseudo-ergodic} if $$X
= L^+ (x) \cup L^- (x)$$ holds. The set of all pseudo-ergodic
elements of $X$ is denoted as $X_{\PE}$.
\end{defini}

We have the following characterization of pseudo-ergodicity of an
$x\in X$, which is not isolated.

\begin{proposi}\label{Characterization-pseudo-ergodic}
Let $(X,T)$ be a dynamical system. For an $x\in X$, which is not
isolated,   the following assertions are equivalent:
\begin{itemize}
\item[$(i)$] The orbit of $x$ is dense.
\item[$(ii)$] The element $x$ is pseudo-ergodic.
\end{itemize}
\end{proposi}
\begin{proof} The implication $(ii)\Longrightarrow  (i)$ is clear (and holds for any $x\in X$).
We now show $(i)\Longrightarrow (ii)$. As $L^+ (x) \cup L^- (x) $ is
closed and invariant under $T$, it suffices to show that it contains
$x$.
As $x$ is not isolated, there exists a sequence $(y_n)$ in $X$
converging to $x$ such that the $y_n$, $n\in\NM$, are pairwise
different and none of them equals $x$. By $(i)$ we can find for any
$y_n$ an index $k_n$ with $ d(T^{k_n} x, y_n) \leq \frac{1}{3}
d(y_n, x)$, where we denote a metric on $X$ by $d$. The assumption
on the $(y_n)$ gives that the index set $\{k_n : n\in \NM\}$ is
infinite. Moreover, $$\lim_{n\to \infty} T^{k_n} x = \lim_{n\to
\infty} y_n  = x.$$ This shows $x\in L^+ (x) \cup L^- (x)$.
\end{proof}

\begin{rem}
It is not hard to see that  the closure of  $\Orb(x)$ equals  $L^+
(x) \cup L^- (x)$ if and only if $x$ belongs to $L^+ (x) \cup L^-
(x)$. Of course, one can easily give examples where $x$ does not
belong to $L^+ (x) \cup L^- (x)$. Consider e.g. the space
$\{0,1\}^\ZM$ of sequences with values in $\{0,1\}$ over $\ZM$ with
the shift operation $T  x (n) = x(n+1)$.  Let $1_{0}$ be  the
characteristic function of $\{0\}$. Then, both $L^+ (1_{0})$ and
$L^- (1_{0})$ consist only  of the  function with value $0$
everywhere. Thus, if we  define $X$ to be the closure of the orbit
of $1_{0}$, the orbit of $1_{0}$ will be dense in $X$ but $1_{0}$
will not be pseudo-ergodic. This shows that  the assumption that $x$
is not isolated is necessary in the previous proposition.
\end{rem}

A dynamical system $(X,T)$ is called {\em minimal} if the orbit
of $x$ is dense in $X$ for  each $x\in X$. We are now going to study
the relationship between minimality and pseudo-ergodicity (of all
elements). We start with the following well-known result. We include
a proof for the convenience of the reader.

\begin{proposi}
\label{ConstSpectr.prop-EquivMin} Let $(X,T)$ be a dynamical system.
Then the following assertions are equivalent.
\begin{itemize}
\item[(i)] The dynamical system $(X,T)$ is minimal.
\item[(ii)] For all $x\in X$ the equation $L^+(x)=X$ holds.
\item[(iii)] For all $x\in X$ the equation $L^-(x)=X$ holds.
\end{itemize}
\end{proposi}

\begin{proof}
We will prove that (i) and (ii) are equivalent. The equivalence of
(i) and (iii) follows similarly. By the obvious inclusion
$L^+(x)\subseteq\overline{\Orb(x)}$ for each $x\in X$ the
implication $(ii)\Rightarrow(i)$ is clear.

\medskip

Assume now (i) is true. Let $x\in X$ and choose an $y\in L^+(x)$.
Such a choice is possible as $L^+ (x)$ is not empty due to
Proposition~\ref{ConstSpectr.prop-PropLimSet}. By the same
proposition the set $L^+(x)$ is closed and invariant under $T$ and
$T^{-1}$. Thus, we conclude $\overline{\Orb(y)}\subseteq L^+(x)$. By
$(i)$ this immediately implies
$$
X=\overline{\Orb(y)}\subseteq L^+(x)\subseteq X
$$
leading to assertion $(ii)$.
\end{proof}

From the preceding considerations, we rather directly obtain the
following characterization of minimality in terms of pseudo-ergodic
elements.

\begin{proposi}\label{Char-minimality-via-pseudo-ergodicity}
Let  $(X,T)$ be a dynamical system. Then the following assertions
are equivalent.
\begin{itemize}
\item[$(i)$] The dynamical system $(X,T)$ is  minimal.
\item[$(ii)$] The equality  $X = X_{\PE}$ holds.
\item[$(iii)$] The set $X_{\PE}$ is closed and non-empty.
\end{itemize}
\end{proposi}
\begin{proof} The implication $(i)\Longrightarrow (ii)$ follows
directly from Proposition \ref{ConstSpectr.prop-EquivMin}. The
implication $(ii)\Longrightarrow (iii)$ is clear. It remains to show
$(iii)\Longrightarrow (i)$. Now, the  set $X_{\PE}$ is clearly
invariant under $T$. Thus, with any $x$ it will contain $\Orb (x)$
and from $(iii)$ and the definition of pseudo-ergodicity we then
infer that $X_{\PE}\supseteq \overline{\Orb (x)} \supseteq L^+(x)\cup L^-(x) = X$. Thus any
element of $x$ is pseudo-ergodic. In particular, the orbit of any
element of $x$ is dense and $(i)$ follows.
\end{proof}

\section{Operators on dynamical systems}
\label{ConstSpectr.sect-OpDynSys}

Given a dynamical system $(X,T)$, a map $A:X\to L(\ell^p,\Pp)$
is called a {\em family of operators over $(X,T)$} if the following conditions hold.
\begin{itemize}
\item[$(i)$] $\sup_{x\in X}\sup_{i,j\in\ZM} |A(x)_{i,j}|<\infty$.\hfill\textit{(Uniform boundedness)}
\item[$(ii)$] $A(Tx)=U^{-1}A(x)U$ for all $x\in X$.\hfill\textit{(Equivariance)}
\item[$(iii)$] The map $x\mapsto A(x)_{i,j}$ is continuous for each $i,j\in\ZM$.\hfill\textit{(Continuity)}
\end{itemize}

Recall that $U$ is the shift operator, acting as an isometric bijection on our space $\ell^p$.
%
%
The boundedness assumption $(i)$ follows, via the uniform boundedness principle,
directly from weak continuity of the map $A$. More specifically, we have the following result.

\begin{proposi}
Let  a dynamical system $(X,T)$, 
$A:X\to L(\ell^p,\Pp)$ be given such that the following
conditions hold:
\begin{itemize}
\item $A(Tx)=U^{-1}A(x)U$ for all $x\in X$.\
\item The map $A$ is continuous with respect to the weak operator
topology.
\end{itemize}
Then, $A$ is  a family of 
operators over $(X,T)$.
\end{proposi}
\begin{proof}
Condition $(ii)$ of the preceding definition is satisfied by
assumption. Condition $(iii)$ can be inferred directly from the
continuity in the weak operator topology (as the map $L(\ell^p) \longrightarrow
\CM, B\mapsto B_{i,j}$ is 
continuous with respect to the weak operator topology
for each
$i,j\in\ZM$). As for $(i)$, we note that the uniform boundedness
principle together with compactness of $X$ gives that the family
$(A(x))_{x\in X}$ is bounded with respect to the norm of $L(\ell^p)$
(which is the usual operator norm). This directly  gives $(i)$.
\end{proof}

For later use, we also note the following simple consequence of the
definition.

\begin{proposi}
\label{ConstSpectr.prop-ShiftMatrixElem} Let $(X,T)$ be a dynamical
system, 
$A:X\to L(\ell^p,\Pp)$ a family of
operators over $(X,T)$. Then,
$$
A(T^nx)_{i,j}=A(x)_{i+n,j+n},\qquad i,j,n\in\ZM,\; x\in X.
$$
\end{proposi}

\begin{proof}
It suffices to consider the cases $n=1$ and $n=-1$. (Then, the
remaining statements follow easily by induction). Let
$e_j:\ZM\to\CM$ be defined by $e_j(i):=\delta_{j,i}$ where $\delta$
denotes the Kronecker delta. For $n=1$ a short computation shows
$$
A(Tx)_{i,j} = \left(A(Tx)e_j\right)(i) = \left(U^{-1}A(x)U e_j \right)(i)
= \left( A(x)e_{j+1}\right)(i+1)
= A(x)_{i+1,j+1}
$$
for $i,j\in\ZM$ and $x\in X$. The case $n=-1$ can be treated similarly.
\end{proof}

As a bounded operator $A$ on $\ell^p$ always satisfies
$\sup_{i,j\in\ZM}|A_{i,j}|<\infty$, it is natural to ask whether
condition $(i)$ on the uniform boundedness of a family of 
operators
can be relaxed. This is indeed possible as discussed in the
following proposition.

\begin{proposi}
\label{ConstSpectr.prop-SuffCondFamBandOp} Let $(X,T)$ be a
dynamical system with one dense orbit in $X$. Consider a map $A:X\to
L(\ell^p,\Pp)$ such that $A(Tx)=U^{-1}A(x)U$ holds for all $x\in X$
and the map $x\mapsto A(x)_{i,j}$ is continuous for each
$i,j\in\ZM$. Then, $A$ is a family of
operators over $(X,T)$.
\end{proposi}

\begin{proof}
We need to show that $\sup_{x\in X}\sup_{i,j\in\ZM} |A(x)_{i,j}|$ is
finite. Let $y\in X$ with dense orbit $\Orb(y)$ be given. Since
$A(y)\in L(\ell^p)$ we know that $\sup_{i,j\in\ZM}
|A(y)_{i,j}|<\infty$. Furthermore, by
Proposition~\ref{ConstSpectr.prop-ShiftMatrixElem}, we get
$\sup_{i,j\in\ZM}|A(x)_{i,j}|=\sup_{i,j\in\ZM}|A(y)_{i,j}|$ for all
$x\in\Orb(y)$. As $\Orb(y)\subseteq X$ is dense and
$A(\cdot)_{i,j}:X\to\CM,\; i,j\in\ZM$ is continuous it follows
$$
\sup_{x\in X}\sup_{i,j\in\ZM} |A(x)_{i,j}| =\sup_{x\in
\Orb(y)}\sup_{i,j\in\ZM} |A(x)_{i,j}|
=\sup_{i,j\in\ZM}|A(y)_{i,j}|<\infty,
$$
which means that $A$ also satisfies the uniform boundedness
condition.
\end{proof}

\section{Bringing it all together: The main
result}\label{section-main} In this section we combine the
considerations and concepts of the previous section to state and
prove our main result.

\bigskip

\begin{proposi}
\label{ConstSpectr.prop-CharSigOp} Let $(X,T)$ be a dynamical
system, 
$A:X\to L(\ell^p,\Pp)$ a family of
operators over $(X,T)$. Then, the equation
$$
\sop(A(x))=\left\{ A(y)\;:\; y\in L^+(x)\cup L^-(x) \right\}
$$
holds for all  $x\in X$.
\end{proposi}

\begin{proof}
We first show ``$\subseteq$''. Let $x\in X$ and $B\in\sop(A(x))$ be
given. Then, there exists a sequence $(h_k)_{k\in\NM}$ of integers
with $|h_k|\to\infty$ as $k\to\infty$, such that
$$
B_{i,j}=\lim\limits_{k\to\infty}
A(x)_{i+h_k,j+h_k}=\lim\limits_{k\to\infty} A(T^{h_k}x)_{i,j},
\qquad i,j\in\ZM,
$$
where the second equality follows from Proposition
\ref{ConstSpectr.prop-ShiftMatrixElem}. Since $ |h_k|\to\infty,\;
k\to\infty$,  and $X$ is compact we can select a subsequence
$(h_{k_j})_{j\in\NM}$ such that $(T^{h_{k_j}}x)_{j\in\NM}$ is
convergent to $y\in X$ and $(h_{k_j})_{j\in\NM}$ tends to $\infty$
or $-\infty$. Thus, $y\in L^+(x)\cup L^-(x)$. For $i,j\in\ZM$, using
the continuity of $A(\cdot)_{i,j}:X\to\CM$, the equation
$B_{i,j}=A(y)_{i,j}$ follows with $y\in L^+(x)\cup L^-(x)$. Hence,
$B$ is an element of the set $\left\{ A(y)\;:\; y\in L^+(x)\cup
L^-(x)\right\}$.

\medskip

We now turn to proving ``$\supseteq$''. For $x\in X$ and $y\in
L^+(x)\cup L^-(x)$ there is a sequence $(h_k)_{k\in\NM}$ tending to
$\infty$ or $-\infty$ such that $\lim_{k\to\infty} T^{h_k}x=y$.
Then, for $i,j\in\ZM$, the continuity of $A(\cdot)_{i,j}:X\to\CM$
and Proposition \ref{ConstSpectr.prop-ShiftMatrixElem} imply
$$
A(y)_{i,j}=\lim\limits_{k\to\infty}
A(T^{h_k}x)_{i,j}=\lim\limits_{k\to\infty} A(x)_{i+h_k,j+h_k}
$$
leading to $A(y)\in\sop(A(x))$.
\end{proof}

As a consequence we immediately deduce the following lemma. On the
technical level this lemma is the crucial  ingredient in the proof
of our main result.

\begin{lemma}
\label{ConstSpectr.lem-CharMinSigOp} Let $(X,T)$ be a  dynamical
system and  $A:X\to L(\ell^p,\Pp)$ a family of
operators over $(X,T)$. Then,
$$
\sop(A(x))=\left\{ A(y)\;:\; y\in X \right\}
$$
holds for all $x\in X_{\PE}$. In particular, for any $x\in X_{\PE}$
the operator  $A(x)$ is self-similar (i.e. $A(x)\in\sop(A(x))$).
\end{lemma}

\begin{proof}
This is a direct consequence of
Proposition~\ref{ConstSpectr.prop-CharSigOp} and the definition of
$X_{\PE}$.
\end{proof}

Now we are able to prove that the spectrum is constant and agrees
with the essential spectrum for any pseudo-ergodic element.

\begin{theo}
\label{ConstSpectr.theo-MinConstSpectr} Let $(X,T)$ be a  dynamical
system and  $A:X\to L(\ell^p,\Pp)$  a family of
operators over $(X,T)$. Set
$$
\Sigma:=\underset{x\in X}{\bigcup}\spec(A(x)).
$$
Then
$$
\spec(A(x))=\specess(A(x))=\Sigma
$$
holds for all $x\in X_{\PE}$.
\end{theo}

\begin{rem}
A similar result was proven for so-called pseudo-ergodic operators
in \cite{Davies2001b}. In fact, the above result is a generalization
of the result of \cite{Davies2001b} (in the case $\varGamma = \ZM$).
Details are discussed below in Section \ref{section-example-Davies}.
We note already here that every pseudo-ergodic operator is
self-similar and for pseudo-ergodic operators the operator spectrum
$\sop(A)$ is very large.
\end{rem}

\begin{proof}[Proof of Theorem \ref{ConstSpectr.theo-MinConstSpectr}]
Let $x\in X_\PE$. By Theorem \ref{ConstSpectr.theo-EssSpectLocInfin} and
Lemma~\ref{ConstSpectr.lem-CharMinSigOp} it follows
$$
\spec(A(x))\supseteq \specess(A(x)) \supseteq \underset{B\in\sop(A(x))}{\bigcup}\spec(B) =
\underset{y\in X}{\bigcup}\spec(A(y)) = \Sigma.
$$
As $x\in X$ appears in the rightmost union, we obtain
the assertion.
\end{proof}

\begin{coro}\label{main-minimal} Let $(X,T)$ be a  minimal dynamical
system, 
$A:X\to L(\ell^p,\Pp)$ a family
of 
operators over $(X,T)$. Set
$$
\Sigma:=\underset{x\in X}{\bigcup}\spec(A(x)).
$$
Then
$$
\spec(A(x))=\specess(A(x))=\Sigma
$$
holds for all $x\in X$.
\end{coro}
\begin{proof} Due to minimality we have $X_{\PE} = X$ by Proposition
\ref{Char-minimality-via-pseudo-ergodicity}. Thus, the statement
follows directly from the previous theorem.
\end{proof}

\begin{rem}
In particular, the corollary  gives that the spectrum of $A(x),\;
x\in X$ is constant and agrees with the essential spectrum for all
operators $A(x),\; x\in X$. As discussed in the introduction, this
is well-known in the case $p=2$ whenever $A(x),\; x\in X$,  is
self-adjoint. For such operators the proof of constancy of the
spectrum relies on a semi-continuity property of the spectrum found
e.g. in \cite[Theorem VIII.24 (a)]{ReedSimon1}. This semi-continuity
does not apply in our case. Indeed, we can consider the example
given in \cite[Exampe IV.3.8]{Kato1980}: For $c\in\RM$ let
  \[A_c:=\begin{pmatrix}
      \ddots &  & & &\\
      \ddots & 0 & & & \\
      & 1 & 0 & & & \\
      & & c & 0 & & \\
      & & & 1 & 0 & \\
      & & &   & \ddots & \ddots
         \end{pmatrix}.\]
Then, for $c\neq 0$ we have $\spec(A_c) = \{z\in\CM:\; |z|=1\}$.
However, $\spec(A_0) = \{z\in\CM:\; |z|\leq 1\}$, although we have
strong convergence of $A_c$ to $A_0$ and in fact even norm
convergence $\|A_c-A_0\|\to 0$ as $c\to 0$. Thus, the corollary
gives a new result even for $p=2$ whenever $A$ is not selfadjoint.
\end{rem}

\begin{rem}
For a minimal dynamical system $(X,T)$ 
and a
family of operators $A:X\to\Ww$, a combination of Proposition
\ref{prop:Spectrum_Wiener} and Theorem
\ref{ConstSpectr.theo-MinConstSpectr} states that for each $x\in X$
and any $\lambda\in\spec(A(x))$ there is $y\in X$ with a generalized
bounded eigenfunction to  the eigenvalue $\lambda$. So one might ask
whether there exists a generalized bounded eigenfunction to $A(x)$
itself for $\lambda\in\spec(A(x))$. We consider this an interesting
question.
\end{rem}

\section{pseudo-ergodic operators over $\ZM$}\label{section-example-Davies}
In this section we discuss shortly how a main result of
\cite{Davies2001b} is a special case of  Theorem
\ref{ConstSpectr.theo-MinConstSpectr}.

\bigskip

Consider a compact subset $S$ of the complex plane  with the induced
topology. Let $X := S^\ZM$ with the product topology and define $T$
via
$$Tx (n) = x(n+1).$$
Then, $(X,T)$ is a dynamical system, known as \textit{shift over
$S$}. In this situation Davies \cite{Davies2001b} calls an $x\in X$
pseudo-ergodic if for any $\varepsilon >0$, any $n\in \NM$ and any
$(y_1,\ldots, y_n)\in S^n$ there exists a $k\in\ZM$ with
$$ \|(x(k+1), \ldots, x(k+n)) - (y_1,\ldots, y_n) \| \leq
\varepsilon.$$ (Here, $\|\cdot\|$ denotes the Euclidean norm in
$\CM^n$.) Then, it is not hard to see (compare also
\cite{Lindner2006}) that $x$ is pseudo-ergodic in the sense of
Davies if and only if $\Orb (x)$ is dense in $X$. This, in turn,  is
equivalent to $x$ being pseudo-ergodic in the sense of our
definition. Indeed,  if $S$ consists of at least two elements,  then
$x$ can not be an isolated element of $X$ and, hence,  Proposition
\ref{Characterization-pseudo-ergodic} gives the desired equivalence.
If $S$ consists of only one element then $X$ consists of one
element only and this element is clearly  pseudo-ergodic both in the
sense of our definition and the definition of Davies. Thus, our
setting contains the setting of \cite{Davies2001b} and hence our
main result, Theorem \ref{ConstSpectr.theo-MinConstSpectr},
generalizes the corresponding result of \cite{Davies2001b}.

One further remark may be in order here: The setting of
\cite{Davies2001b} is not restricted to shifts with respect to $\ZM$.
Instead rather general discrete groups are allowed for. In this
sense, the results of Davies are still somewhat more general than
ours.

\section{Examples of minimal dynamical
systems}\label{section-examples} In this section we discuss some
examples where Corollary \ref{main-minimal} can be applied. These
are Sturmian models, quasiperiodic models, and almost periodic
models. For all of these classes of models  associated selfadjoint
Schr\"odinger type operators attracted considerable attention.

\medskip

\subsection*{Sturmian models}
We consider $\{0,1\}$ with the discrete topology and equip
$\{0,1\}^{\ZM}$ with the product topology and the 'shift' operation
$T$ given by $T x (n) = x (n+1)$. In this way $(\{0,1\}^\ZM,T)$
becomes a dynamical system.  Consider now $\alpha\in (0,1)$
irrational and define
$$V_\alpha : \ZM\longrightarrow \{0,1\}, V_\alpha (n) :=
1_{(1-\alpha,1]} (n\alpha \; \mbox{mod}\; 1).$$ Let  $X_\alpha$ be
the closure of the orbit of $V_\alpha$ in $\{0,1\}^{\ZM}$. Then,
$X_\alpha$ is invariant under $T$ and $(X_\alpha,T)$ is a dynamical
system. It is known as \textit{Sturmian dynamical system} or \emph{Sturmian subshift (with
rotation number $\alpha$)}. Sturmian dynamical systems  are minimal.

Sturmian  dynamical systems play an important role in the
investigations of  a special type of solids discovered in 1982 and
later called  quasicrystals, see  \cite{Dam,DEG,Lenz1} for surveys
on such operators and further references. In fact, the most
prominent model in the investigation of quasicrystals is the
Sturmian subshift with rotation number $\alpha = \mbox{golden
mean}$. This is known as \textit{Fibonacci model}.

Sturmian dynamical systems   have the following complexity feature:
For any natural number $n$ the set
$$\{(\omega (k+1)\ldots \omega (k+n)) : \omega\in X_\alpha, k\in
\ZM\}$$ has exactly $n+1$ elements. In fact, this latter property
even characterizes Sturmian dynamical systems (among the minimal
subshifts over $\{0,1\}$). As we do not need this, we refrain from
further discussion.

For the quantum mechanical treatment of conductance properties of
quasicrystals (in one dimension) mostly Sturmian models are
considered. There, one considers the function
$$\delta : X_\alpha \longrightarrow \{0,1\}, \delta (x) := x(0).$$
This is then used to define, for any $x\in X_\alpha$, the
multiplication operator $V_x$ on $\ell^2 $ satisfying
$$(V_x f) (n) := \delta (T^n x) f (n) = x (n) f(n).$$
The conductance properties are then encoded in the spectral theory
of the family of operators
$$H_x := U + U^{-1} + \lambda  V_x : \ell^2 \longrightarrow \ell^2$$
for $x\in X_\alpha$. Here, $\lambda \neq 0$ is arbitrary and  $U$ is
the shift $Uf (n) = f(n-1)$ (which was already discussed above).
This is a family of selfadjoint operators, see the mentioned
references \cite{Dam,Lenz1} for further discussion.

However, we could easily go over to a family of non-selfadjoint
operators by considering e.g. the family $A : X_\alpha
\longrightarrow \Band$ with
$$ A(x) := U + \lambda V_x,$$
$x\in X_\alpha,$ for $\lambda \neq 0$.

\subsection*{Quasiperiodic   models}
We consider for $n\in \NM$ the set $\TM:= \RM^n / \ZM^n$. For $\beta
\in \RM^n$ we then define the `rotation' action $R$ on $\TM$ by
$$R : \TM \longrightarrow \TM, \; R (v + \ZM^n) := v + \beta + \ZM^n.$$
Then, $(\TM, R)$ is a dynamical system. If the entries of $\beta$
are rationally independent, then this dynamical system is minimal.
Consider now a continuous function $\varphi : \TM \longrightarrow \CM$.
Then, this function gives, for any $x\in \TM$, rise to the
multiplication operator $V_x$ acting on $\ell^2 $ via $(V_x f)
(n) = \varphi (R^n x) f(n)$. This then induces the family

$$H_x := U + U^{-1} + \lambda  V_x : \ell^2\longrightarrow \ell^2 $$
for $x\in \TM$. Here, $\lambda \neq 0$ is arbitrary and  $U$ is,
again,  the shift $Uf (n) = f(n-1)$. If $\varphi$ is real-valued,
this is a family of selfadjoint operators. They are known as
\textit{discrete quasiperiodic Schr\"odinger operators}.

The most prominent example is the case $n =1$, $\varphi (x + \ZM) =
\cos (2\pi x)$, and $\beta$ irrational. The associated operator is
known as \textit{almost Mathieu operator}. It has been studied
immensely, see e.g. the surveys \cite{Dam2,Jit,Las}.

Again, we can  easily go over to a family of non-selfadjoint
operators by considering e.g.\ the family $A : X_\alpha
\longrightarrow \Band$ with
$$ A(x) := U + \lambda V_x,$$
$x\in X_\alpha,$ for $\lambda \neq 0$.

\subsection*{Almost periodic models}
Consider again the shift $(X,T)$ over a compact subset $S\subseteq \CM$
as in Section \ref{section-example-Davies} (i.e.\ $X:=S^\ZM$ with the product topology and $Tx(n):=x(n+1)$).
Then, $x\in X$ is called \emph{almost periodic} if $\Orb(x)$ is relatively compact in $\ell^\infty$.
In this case, let $X_x$ be the closure of $\Orb(x)$ in $X$, which due to the relative compactness in $\ell^\infty$ coincides with the closure in $\ell^\infty$ (sometimes called the \emph{hull} of $x$).
Clearly, $(X_x,T)$ is a minimal dynamical system.
For $y\in X_x$ let
\[H_y := U + U^{-1} + M_y : \ell^2\longrightarrow \ell^2,\]
where $U$ is again the shift $Uf(n) = f(n-1)$ and
$M_yf(n):=y(n)f(n)$ is the multiplication operator induced by $y$.
For $S\subseteq \RM$, the operators  $H_y$  are selfadjoint and known
as \textit{discrete almost periodic Schr\"odinger operators}.

Starting with \cite{Favard1927}, almost periodic models (in
arbitrary dimensions)  were studied intensively,  see
e.g.~\cite{PF,Corduneanu1989,Shubin1978}. The first systematic study
of almost periodic Schr\"odinger operators was given in
\cite{AvronSimon1981,AvronSimon1983}.

Note that almost periodic models generalize the above-mentioned
quasiperiodic models, as can easily be inferred from the continuity
of  $\TM\longrightarrow \ell^\infty$, $v\mapsto (\varphi(R^n
v))_{n\in\ZM}$.

\section{Some further aspects}
In this section  we discuss how a family of 
operators
can be seen as a dynamical system  itself (under a weak continuity
assumption)  and how this dynamical system is related to the
original  dynamical system.

\bigskip

Let a dynamical system $(X,T)$ and a family of 
operators 
$A:X\to L(\ell^p,\Pp)$ be given.
Assume that $A$ is continuous with respect to the weak operator
topology. (This is for example the case if $A$ takes values in
$\Band$ and there is a uniform upper bound for the band width.)
Then, the  set $X_A:= \{A(x) :  x\in X\}$ (with
the weak operator topology) is compact and the map
$$X\longrightarrow X_A,
x\mapsto A(x),$$ is continuous. Moreover, the map
$$Ad_U:  B\mapsto U^{-1} B U$$
is a homeomorphism of this set.  Thus, $(X_A, Ad_U)$ is a dynamical
system. In this way, any such family of 
operators
comes with two dynamical systems viz the system $(X,T)$ and the
system $(X_A, Ad_U)$. We will now study the relationship between
these two dynamical systems.

\smallskip

Clearly, $(X_A, Ad_U)$ is a factor of $(X,T)$ via the map $A$, i.e.\
$A$ gives a continuous surjective map from $X$ to $X_A$ intertwining
the actions of $T$ and $Ad_U$.

Moreover, if the  family of 
operators
$A:X\to L(\ell^p,\Pp)$ is furthermore  injective, both systems
are conjugate (i.e.\ there exists a homeomorphism between them
intertwining the actions of $T$ and $Ad_U$.)  Thus, in that case one
of the dynamical systems is minimal if and only if the other one is
minimal as well. We lose this  as soon as we do not have the
injectivity of $A:X\to L(\ell^p,\Pp)$. This can be seen by the
next example.

\begin{exam}{\rm
Let $(X,T)$ be a an arbitrary non-minimal  dynamical system and
$B\in\Band$ be a band 
operator satisfying $U^{-1}BU=B$. Then, the constant
family $A(x):=B,\; x\in X$ is a family of 
operators
over $(X,T)$ and continuous with respect to the weak operator
topology. Then,  $(X_A, Ad_U)$ is minimal (as it only consists of
one point). Clearly, $A:X\to\Band$ is not injective in this case. }
\end{exam}

\medskip

As we have seen in Lemma~\ref{ConstSpectr.lem-CharMinSigOp}
minimality of $(X,T)$ implies that $A$ is self-similar. One might
ask if the converse holds as well. Clearly, the injectivity of
$A:X\to L(\ell^p,\Pp)$ is a necessary condition. But if $A$ is
injective and self-similar, it is also not necessarily  true that
$(X,T)$ is minimal. This can be seen by the following example.

\begin{exam}{\rm
Let $(X_1,T)$ and $(X_2,T)$ be two different minimal subshifts with
alphabet $\as_1$ and $\as_2$ such that $\as_1\cap\as_2=\varnothing$. Consider two bijective maps
$\Phi_1:\as_1\to\{1,\ldots,|\as_1|\}$ and
$\Phi_2:\as_2\to\{|\as_1|+1,\ldots,|\as_1|+|\as_2|\}$. Define two
families of 
operators $A^{(1)}:X_1\to\Band$ and
$A^{(2)}:X_2\to\Band$ by
\begin{align*}
A^{(k)}(x)_{i,i}&:=\Phi_k(x(i))\quad\text{and}\quad
A^{(k)}(x)_{i,j}:=0,\qquad\qquad x\in X_k,\; i,j\in\ZM,\; i\neq j,\;
k=1,2.
\end{align*}
Then, $X:=X_1\cup X_2$ together with $T$ forms a dynamical system.
By $A(x):=A^{(k)}(x),\; x\in X_k$ we define a map $A:X\to\Band$. By
minimality of $(X_k,T),\; k=1,2$ it follows that $A:X\to\Band$ is a
family of 
operators and each $A(x),\; x\in X$ is self-similar.
Moreover, $A$ is injective. On the other hand, it is immediate to
see that $(X,T)$ is not minimal.}
\end{exam}

\end{document}